\documentclass[a4paper,10pt]{amsart}
\usepackage[utf8]{inputenc}

\usepackage{graphicx,amssymb,amsmath,amsfonts,amsthm,amscd,hyperref,textcomp,relsize}

\newtheorem{proposition}{Proposition}

\newtheorem{lemma}{Lemma}
\newtheorem{corollary}{Corollary}
\newtheorem{theorem}{Theorem}

\theoremstyle{definition}\newtheorem{remark}{Remark}

\newcommand{\RR}{\mathbb R}
\newcommand{\QQ}{\mathbb Q}
\newcommand{\Fp}{{\mathbb F}_p}
\newcommand{\Fq}{{\mathbb F}_q}

\newcommand{\Ql}{\bar{\mathbb Q}_\ell}
\newcommand{\Dbc}{D^b_c}

\newcommand{\FF}{\mathcal F}
\newcommand{\GGG}{\mathcal G}
\newcommand{\HH}{\mathcal H}
\newcommand{\LL}{\mathcal L}
\newcommand{\R}{\mathrm R}

\newcommand{\AAA}{\mathbb A}

\newcommand{\Gm}{{\mathbb G}_m}

\newcommand{\ZZ}{\mathbb Z}

\newcommand{\Spec}{\mathrm{Spec}}

\newcommand{\vv}{\mathbf v}

\newcommand{\Vol}{\mathrm{Vol}}
\newcommand{\nullsp}{\mathrm{null}}

\title{On the rank of hypergeometric sheaves on higher dimensional tori}
\author{Antonio Rojas-Le\'on}
\address{Departamento de Álgebra \\
         Facultad de Matemáticas \\
         c/Tarfia s/n, 41012 Sevilla, SPAIN}
\email{arojas@us.es}

\begin{document}

\begin{abstract}
We prove an explicit formula for the generic rank of hypergeometric perverse sheaves on higher dimensional tori over a finite field, as defined by Gabber and Loeser. This allows to give an explicit estimate for the corresponding hypergeometric character sums for almost all primes $p$.
\end{abstract}

\maketitle

\renewcommand{\thefootnote}{}
\footnote{ORCID code: 0000-0003-1683-9487}
\footnote{Partially supported by PID2024-156912N, financed by MICIU/AEI/10.13039/501100011033 and FEDER, UE}

\section{Introduction}

Let $\Fq$ be a finite field of characteristic $p$, and $\Gm^r:={\mathbb G}^r_{m,\Fq}$ the split $r$-dimensional torus over $\Fq$. We will fix a prime $\ell\neq p$, and work on the derived category $\Dbc(\Gm^r,\Ql)$ of $\ell$-adic sheaves on $\Gm^r$.

In their fundamental article \cite{gabber1996faisceaux}, Gabber and Loeser defined and studied the properties of the so-called hypergeometric sheaves on $\Gm^r$, the higher dimensional analogues of the hypergeometric shaves whose theory (both in the $\ell$-adic and in the $D$-module categories) was developed by Katz in his book \cite{katz1990esa}.

These are perverse sheaves on $\Gm^r$, which (up to extension by negligible objects) can be characterized by having Euler characteristic one. They are constructed as a succesive convolution of building blocks, each of which is of the form $\iota_{!}(\LL_{\psi}\otimes\LL_\chi)[1]$, where $\iota:\Gm\hookrightarrow\Gm^r$ is a non-constant torus homomorphism, $\LL_\psi$ is the restriction to $\Gm$ of the Artin-Schreier sheaf associated to a fixed additive character $\psi:\Fq\to\Ql^\times$ (which we assume to be the pull-back of a character of $\Fp$, so it is $p$-Frobenius invariant) and $\LL_\chi$ is the Kummer sheaf on $\Gm$ associated to some multiplicative character $\chi:\Fq^\times\to\Ql^\times$ (see \cite[1.7]{deligne569application} for their definitions).

As explained in \cite[1.6,3.7.3,8.1.2]{gabber1996faisceaux}, there are different concepts of ``convolution'' in $\Gm^r$, depending on which direct image one takes to define it:
$$
K\ast_\ast L:=\R\mu_\ast(K\boxtimes L)$$
$$
K\ast_! L:=\R\mu_!(K\boxtimes L)$$
$$
K\ast_{int} L:=\mathrm{im}\left(\R\mu_!(K\boxtimes L)\to \R\mu_\ast(K\boxtimes L)\right)
$$
where $\mu:\Gm^r\times\Gm^r\to\Gm^r$ is the multiplication map. Consequently, there are at least three different types of hypergeometric objects. One way to solve this ambiguity \cite[3.6.4]{gabber1996faisceaux} is to work on the quotient category of $\Dbc(\Gm,\Ql)$ by the thick subcategory of negligible objects (ie. objects whose perverse cohomology sheaves have Euler characteristic zero). However, equivalent (perverse) objects in this quotient category may have different ranks, so working in the quotient category is not appropriate if one needs to deal with arithmetic applications of these objects.

Since we are interested in the these arithmetic applications, and the direct image with compact support is more suited for them, in this article we will always be working with the ``compactly supported'' convolution, that is, $K\ast L$ will always mean $K\ast_!L$.

Fix $n$ multiplicative characters $\chi_1,\ldots,\chi_n:\Fq^\times\to\Ql^\times$ (equivalently, a multiplicative character $\chi:(\Fq^\times)^n\to\Ql^\times$ given by $\chi(x_1,\ldots,x_n)=\chi_1(x_1)\cdots\chi_n(x_n)$, which we will denote by $\otimes_{i=1}^n\chi_i$), and $n$ non-zero column vectors $\mathbf a_1,\ldots,\mathbf a_n\in\ZZ^r$, which define morphisms $\iota_1,\ldots,\iota_n:\Gm\to\Gm^r$ given by $\iota_j(t)=(t^{a_{1j}},\ldots,t^{a_{nj}})$. Note that we do not assume the $\mathbf a_i$ to be primitive and, in particular, the $\iota_i$ are finite maps but not necessarily closed immersions. Let $A\in\mathcal M_{r,n}(\mathbb Z)$ be the matrix whose columns are ${\mathbf a}_1,\ldots,{\mathbf a}_n$, and denote by $\HH_{A,\chi}\in\Dbc(\Gm^r,\Ql)$ object
$$
\HH_{A,\chi}=\iota_{1!}(\LL_\psi\otimes\LL_{\chi_1})[1]\ast\cdots\ast\iota_{n!}(\LL_\psi\otimes\LL_{\chi_n})[1],
$$
which is a perverse sheaf by \cite[8.1.3]{gabber1996faisceaux}. In fact, by \cite[8.6.1]{gabber1996faisceaux} every perverse sheaf on $\Gm^r$ defined over some finite field with Euler characteristic one is a multiplicative translate of one of this form, up to extension by ``negligible'' sheaves.

By Grothendieck's trace formula, the trace of a geometric Frobenius element acting on the stalk of $\HH_{A,\chi}$ at a geometric point over $\mathbf t=(t_1,\ldots,t_r)\in\Gm^r(\Fq)=(\Fq^\times)^r$ is given by
$$
\mathrm{Hyp}_{A,\chi}(\mathbf t)=(-1)^n\sum_{\substack{x_1^{a_{i1}}\cdots x_n^{a_{in}}=t_i \\ i=1,\ldots,r}}\psi(x_1+\cdots+x_n)\chi_1(x_1)\cdots\chi_n(x_n).
$$
Moreover, by Weil II \cite[Th\'eor\`eme 3.3.1]{deligne1980conjecture}, $\HH_{A,\chi}$ is mixed of weights $\leq n$ (since each factor $\iota_{i!}(\LL_\psi\otimes\LL_{\chi_i})[1]$ is pure of weight $1$). By the general structure of perverse sheaves \cite{beilinson1982faisceaux}, if the support $T$ of $\HH_{A,\chi}$ has dimension $k$, then there exists a non-empty open set $U\subseteq T$ (we will later see that $T$ is a connected subtorus, so $U$ is actually a dense open set) such that $\HH_{A,\chi|U}=\FF[k]$ for some lisse sheaf $\FF$ on $U$, necessarily mixed of weights $\leq n-k$. In particular, if $D$ is the rank of $\FF$, we get an estimate
$$
|\mathrm{Hyp}_{A,\chi}(\mathbf t)|\leq D\cdot q^{\frac{n-k}{2}}
$$
for every $\mathbf t\in U(\Fq)$.

In order for this estimate to be useful in practice, one needs to known (1) an explicit open set $U\subseteq T$ on which $\HH_{A,\chi}$ is a shifted lisse sheaf and (2) the rank of such sheaf. These are the two questions that we attempt to answer in this article.

To state the main theorem, we define two positive integers $d(A),v(A)$ associated to the matrix $A$. $d(A)$ is the product of the invariant factors of $A$ -- equivalently, the greatest common divisor of the $k\times k$ minors of $A$, where $k=\mathrm{rank}(A)$. If $k=n$, we set $v(A)=1$. Otherwise, let $P_0\in\mathcal M_{n,n-k}$ be a matrix whose columns form a basis for $\nullsp(A)$, the right nullspace of $A$, and let $\Delta$ be the convex hull in $\RR^{n-k}$ of the finite set containing the rows of $P_0$ and the origin. Then $v(A):=(n-k)!\Vol(\Delta)$. Since the matrix $P_0$ has full rank $n-k$, $v(A)>0$. And $v(A)$ does not depend on the choice of a basis for $\nullsp(A)$: for any other choice, the corresponding matrix $P'_0$ would be obtained from $P_0$ by right multiplication by an invertible matrix in $\mathcal M_{n-k}(\ZZ)$, so the corresponding polytopes $\Delta$ and $\Delta'$ would be related by a unimodular automorphism of $\RR^{n-k}$, which preserves volumes. Finally, for any positive integer $N$ we define $N_p$ and $N_{\hat p}$ to be the largest power of $p$ that divides $N$ and its prime to $p$ part, so that $N=N_pN_{\hat p}$.

The main result of this article is:

\begin{theorem}\label{main}
 Let $A\in\mathcal M_{r,n}(\ZZ)$ be a non-zero matrix of rank $k$. There exist a finite set of primes $S$, a $k$-dimensional subtorus $T\subseteq{\mathbb G}^r_{m,\ZZ[1/S\ell]}$ and a closed subscheme $Z\subseteq T$ such that for every finite field $\Fq$ of characteristic $p\not\in S$, every specialization $\Spec(\Fq)\to\Spec(\ZZ[1/S\ell])$ and every character $\chi:(\Fq^\times)^n\to\Ql^\times$, the perverse sheaf $\HH_{A,\chi}$ is supported on $T_{\Fq}$, the open subset $U_{\Fq}:=(T\backslash Z)_{\Fq}$ is non-empty and the restriction of $\HH_{A,\chi}$ to $U_{\Fq}$ is a shifted lisse sheaf of rank $d(A)_{\hat p}v(A)$. All of $S$, $T$ and $Z$ can be explicitly computed in terms of $A$.
\end{theorem}

See section \ref{sec-crit} for the precise definitions of $S,T$ and $Z$. We can then give an explicit estimate for the hypergeometric sum:

\begin{corollary}
 With the previous notation, for any $\mathbf t\in U(\Fq)$ we have the estimate
 $$
 |\mathrm{Hyp}_{A,\chi}(\mathbf t)|\leq d(A)_{\hat p}v(A)q^{\frac{n-k}{2}}.
 $$
\end{corollary}

\begin{remark}Note that, even though the generic rank of $\HH_{A,\chi}$ is independent of $\chi$, its irreducibility and its actual weights may depend on it. This is already visible in dimension one: if one takes $A=(1\;-1)$ with $\chi_1\chi_2$ non-trivial, then $\HH_{A,\chi}$ is an irreducible perverse sheaf, pure of weight 2 (it is actually isomorphic to $\LL_{\chi_1(t)}\otimes\LL_{\overline{\chi_1\chi_2}(t+1)}\otimes g(\chi_1\chi_2)^{deg}[1]$, where $g(\chi_1\chi_2)$ denotes the Gauss sum). However, if $\chi_1\chi_2$ is trivial then $\HH_{A,\chi}$ is isomorphic to the direct sum of $\LL_{\chi_1}[1]$ (negligible of weight 1) and a weight 2 punctual sheaf supported at $-1$. In both cases, the generic rank is $1$, but the weights are different. In particular, theorem \ref{main} would be false in general if we used middle convolution (which ``discards'' the negligible components) instead of the compactly supported one.
\end{remark}

The main ingredient in the proof of theorem \ref{main} is the theory of $\ell$-adic GKZ hypergeometric sheaves developed by Fu in \cite{fu2016}. This is the $\ell$-adic counterpart of the GKZ hypergeometric differential systems defined by Gelfand, Kapranov and Zelevinsky \cite{gkz89} and parameterizes the twisted toric exponential sums previously studied by Adolphson and Sperber \cite{AS93} and Fu \cite{fu2009weights}. Let $B=(b_{ij})\in\mathcal M_{m,n}(\ZZ)$ be a matrix and $F:\Gm^m\times\Gm^n\to\AAA^1$ the map given by
$$
(t_1,\ldots,t_m;x_1,\ldots,x_n)\mapsto\sum_{i=1}^m t_ix_1^{b_{i1}}\cdots x_n^{b_{in}}.
$$
Fix also a multiplicative character $\chi:(\Fq^\times)^n\to\Ql^\times$. Then the GKZ hypergeometric sheaf is the object in $\Dbc(\Gm^m,\Ql)$ defined by
$$
{\GGG}_{B,\chi}=\R\pi_{1!}(\pi_2^*\LL_\chi\otimes F^*\LL_\psi)[m+n],
$$
where $\pi_1:\Gm^m\times\Gm^n\to\Gm^m$ and $\pi_2:\Gm^m\times\Gm^n\to\Gm^n$ are the projections and $\LL_\chi$ is the Kummer sheaf on $\Gm^n$ associated to $\chi$ (Note that in \cite{fu2016} the object is defined on the entire $\AAA^m$, while we only consider its restriction to $\Gm^m$). It is a perverse sheaf \cite[Theorem 0.3]{fu2016} with generic rank $n!\Vol(\Delta)$, where $\Delta\subseteq\RR^n$ is the convex hull of the set consisting of the rows of $B$ and the origin, provided that there is a dense open set of $(t_1,\ldots,t_m)$ in $\Gm^m$ such that the Laurent polynomial $F(t_1,\ldots,t_m;x_1,\ldots,x_n)\in\Fq[x_1^{\pm 1},\ldots,x_n^{\pm 1}]$ is non-degenerate with respect to $\Delta$ \cite[Theorem 0.4]{fu2016}.

In section \ref{vs} we will show that (Gabber-Loeser) hypergeometric sheaves on $\Gm^r$ can be constructed from GKZ hypergeometric sheaves using elementary sheaf operations (restriction and direct image by finite maps). This will allow us to deduce their rank from the rank of the corresponding GKZ sheaf, provided that \cite[Theorem 0.4]{fu2016} can be applied. In section \ref{sec-crit} we will prove that, for a given matrix $A$, the theorem can indeed be applied in all but finitely many characteristics, which can be explicitly determined by a simple criterion. In section \ref{examples} we work out some simple cases and, in particular, check that our results match the already known results in dimension one. Finally, in section \ref{lift} we give an analogue of \cite[Theorem 0.8]{fu2016} which states that, in certain cases, the hypergeometric sheaf $\HH_{A,\chi}$ can be lifted to an object in characteristic zero.

\section{GL hypergeometric vs GKZ hypergeometric sheaves}\label{vs}

In this section, we will give an expression of hypergeometric perverse sheaves on $\Gm^r$ (in the sense of Gabber and Loeser) in terms of GKZ hypergeometric sheaves, thus answering a question posed by Fu in the introduction of \cite{fu2016}. We start with some generalities about morphisms between tori.

Let $A\in\mathcal M_{r,n}(\ZZ)$ be a matrix. Then $A$ defines a morphism of tori $[A]:\Gm^n\to\Gm^r$ given by
$$
[A](s_1,\ldots,s_n)=\left(\prod_{j=1}^n s_j^{a_{1j}},\ldots,\prod_{j=1}^n s_j^{a_{rj}}\right).
$$
The following properties are easy to check:
\begin{enumerate}
 \item $[A+B]=[A]\cdot[B]$ for $A,B\in\mathcal M_{r,n}(\ZZ)$.
 \item $[BA]=[B]\circ[A]$ for $A\in\mathcal M_{r,n}(\ZZ)$, $B\in\mathcal M_{s,r}(\ZZ)$.
 \item $[A\sqcup B]=[A]\otimes[B]$ for $A\in\mathcal M_{r,n}(\ZZ)$, $B\in\mathcal M_{r,m}(\ZZ)$; where $A\sqcup B\in\mathcal M_{r,n+m}$ is the yuxtaposition of $A$ and $B$ and $[A]\otimes[B]:\Gm^{n+m}=\Gm^n\times\Gm^m\to\Gm^r$ is given by $([A]\otimes[B])(s,t)=[A](s)\cdot [B](t)$.
 \item The image of $[A]$ is a connected subtorus of $\Gm^r$ of dimension $\mathrm{rank}(A)$. In particular, $[A]$ is surjective if and only if $\mathrm{rank}(A)=r$.
 \item The kernel of $[A]$ is a subgroup of $\Gm^n$ of dimension $n-\mathrm{rank}(A)$. In particular, $[A]$ is a finite map if and only if $\mathrm{rank}(A)=n$.
 \item If $A\in\mathcal M_n(\ZZ)$ is invertible, then $[A]:\Gm^n\to\Gm^n$ is an automorphism with inverse $[A^{-1}]$.
 
\end{enumerate}

The following lemma gives an alternative description of the hypergeometric object $\HH_{A,\chi}$.

\begin{lemma} Let $[A]:\Gm^n\to\Gm^r$ be the torus homomorphism defined by $A$. Then 
$$\HH_{A,\chi}=\R A_!((\LL_\psi\otimes\LL_{\chi_1})\boxtimes\cdots\boxtimes(\LL_\psi\otimes\LL_{\chi_n}))[n]=
\R [A]_!(\LL_{\psi(x_1+\cdots+x_n)}\otimes\LL_\chi)[n].
$$
\end{lemma}

\begin{proof}
We proceed by induction on $n$, the case $n=1$ being clear by definition. Assuming the $n-1$ case is known, let $A'$ be the matrix formed by the first $n-1$ columns of $A$, and $\chi':=\otimes_{i=1}^{n-1}\chi_i:\Gm^{n-1}\to\Ql^\times$. Note that $\iota_n=[\mathbf a_n]$, where we view $\mathbf a_n$ as a column matrix. Then
$$
\HH_{A,\chi}=\HH_{A',\chi'}\ast\iota_{n!}(\LL_\psi\otimes\LL_{\chi_n})[1]=
$$
$$
=\R [A']_!(\LL_{\psi(x_1+\cdots+x_{n-1})}\otimes\LL_{\chi'})[n-1]\ast\iota_{n!}(\LL_\psi\otimes\LL_{\chi_n})[1]=
$$
$$
=\R \mu_!(\R [A']_!(\LL_{\psi(x_1+\cdots+x_{n-1})}\otimes\LL_{\chi'})\boxtimes\iota_{n!}(\LL_\psi\otimes\LL_{\chi_n}))[n]=
$$
$$
=\R \mu_!(\R([A']\times\iota_n)_!((\LL_{\psi(x_1+\cdots+x_{n-1})}\otimes\LL_{\chi'})\boxtimes(\LL_{\psi}\otimes\LL_{\chi_n})))[n]=
$$
$$
=\R(\mu\circ([A']\times[\mathbf a_n]))_!(\LL_{\psi(x_1+\cdots+x_n)}\otimes\LL_\chi)[n]=
$$
$$
=\R[A'\sqcup\mathbf a_n]_!(\LL_{\psi(x_1+\cdots+x_n)}\otimes\LL_\chi)[n]=
$$
$$
=\R[A]_!(\LL_{\psi(x_1+\cdots+x_n)}\otimes\LL_\chi)[n],
$$
where $\mu:\Gm^r\times\Gm^r\to\Gm^r$ denotes the multiplication map, using property (3) above.
\end{proof}

Let $QAP=D$ be a Smith normal form of $A$, with $Q\in\mathcal M_r(\ZZ),P\in\mathcal M_n(\ZZ)$ invertible and $D\in \mathcal M_{r,n}(\ZZ)$ diagonal. Then $[Q]:\Gm^r\to\Gm^r$ and $[P]:\Gm^n\to\Gm^n$ are automorphisms. Let $k$ be the rank of $A$, and $P_1,P_0$ the submatrices consisting of the first $k$ and last $n-k$ columns of $P$. The columns of $P_0$ form a basis for the (right) nullspace of $A$, $\nullsp(A)$. If $\mathbf d:=(d_1,\ldots,d_k)$ are the invariant factors of $A$ (that is, the non-zero diagonal elements of $D$), then $[D]=[\mathbf d]\pi_k$, where $\pi_k:\Gm^n\to\Gm^k$ is the projection onto the first $k$ coordinates and $[\mathbf d]:\Gm^k\to\Gm^r$ is the finite map $(t_1,\ldots,t_k)\mapsto(t_1^{d_1},\ldots,t_k^{d_k},1,\ldots,1)$ (which can be further decomposed into a purely inseparable map followed by a finite \'etale one and a closed immersion). Note that $[\mathbf d]$ is (geometrically) surjective if and only if $k=r$.

Denote by $p_{ij}$ the entries of the matrix $P$. For every $j=1,\ldots,n$, let $\eta_j=\prod_{i=1}^n\chi_i^{p_{ij}}$. Then we have
\begin{proposition} If $k=n$, the perverse sheaf $\HH_{A,\chi}$ on $\Gm^r$ is isomorphic to
$$[Q]^\ast[\mathbf d]_!\left(\LL_{\otimes_{j=1}^n\eta_j}\otimes[P]^\ast\LL_{\psi(x_1+\cdots+x_n)}\right)[n].$$
If $k<n$, it is isomorphic to 
 $$[Q]^\ast[\mathbf d]_!\left(\LL_{\otimes_{j=1}^k\eta_j}\otimes [P_1]^\ast\GGG_{P_0,\eta}\right)[k-n],$$
 where $\eta=\otimes_{j=k+1}^n\eta_j$.
\end{proposition}
\begin{proof}
We have
$$
\mathcal H_{A,\chi}=\R[A]_!(\LL_{\psi(x_1+\cdots+x_n)}\otimes\LL_\chi)[n]=
$$
$$
=\R[Q^{-1}QAPP^{-1}]_!(\LL_{\psi(x_1+\cdots+x_n)}\otimes\LL_\chi)[n]=
$$
$$
=[Q]^\ast\R[D]_![P]^\ast(\LL_{\psi(x_1+\cdots+x_n)}\otimes\LL_\chi)[n]=
$$
$$
=[Q]^\ast[\mathbf d]_!\R\pi_{k!}[P]^\ast(\LL_{\psi(x_1+\cdots+x_n)}\otimes\LL_\chi)[n]=
$$
$$
=[Q]^\ast[\mathbf d]_!\R\pi_{k!}(\LL_{\psi(\sum_i x_1^{p_{i1}}\cdots x_n^{p_{in}})}\otimes\LL_{\chi([P](x))})[n]=
$$
$$
=[Q]^\ast[\mathbf d]_!\R\pi_{k!}(\LL_{\psi(\sum_i x_1^{p_{i1}}\cdots x_n^{p_{in}})}\otimes\LL_{\otimes_{j=1}^n\eta_j})[n]=
$$
$$
=[Q]^\ast[\mathbf d]_!\left(\LL_{\otimes_{j=1}^k\eta_j}\otimes\R\pi_{k!}(\LL_{\psi(\sum_i x_1^{p_{i1}}\cdots x_n^{p_{in}})}\otimes\LL_{\eta})\right)[n]
$$
by the projection formula, since $[\mathbf d]$ is a finite map and $[P]$ and $[Q]$ are automorphisms. If $k=n$, then $\pi_k$ is the identity, $\eta$ is trivial and we are done. Otherwise, if $[P_1]:\Gm^k\to\Gm^n$ is the torus homomorphism defined by $P_1$, then
$$
\LL_{\psi(\sum_i x_1^{p_{i1}}\cdots x_n^{p_{in}})}\otimes\LL_\eta\cong
$$
$$
\cong([P_1]\times Id)^\ast \LL_{\psi(\sum_i y_i\prod_{j=k+1}^n x_j^{p_ {ij}})}\otimes\LL_{\eta}
$$
where $[P_1]\times Id:\Gm^n=\Gm^k\times\Gm^{n-k}\to\Gm^n\times\Gm^{n-k}$. By proper base change, we then have
$$
\R\pi_{k!}(\LL_{\psi(\sum_i x_1^{p_{i1}}\cdots x_n^{p_{in}})}\otimes\LL_{\eta})[n]\cong
$$
$$
\cong[P_1]^\ast\R\pi_{n!}(\LL_{\psi(\sum_i y_i\prod_{j=k+1}^n x_j^{p_ {ij}})}\otimes\LL_{\eta})[n]=
$$
$$
=[P_1]^\ast\GGG_{P_0,\eta}[k-n],
$$
where $\pi_n:\Gm^n\times\Gm^{n-k}\to\Gm^n$ is the projection onto the first factor (whose coordinates are denoted by $y_1,\ldots,y_n$).
\end{proof}

Note that $[P_1]$ is a closed immersion: more precisely, it is the restriction of the automorphism $[P]$ of $\Gm^n$ to the subtorus given by $x_{k+1}=\ldots=x_n=1$, which can be identified with $\Gm^k$.

Let $\Delta$ be the convex hull in $\RR^{n-k}$ of the rows of $P_0$ plus the origin. Denote by $\sigma_i(\mathbf s)=\sum_{j=1}^k s_j^{p_{ij}}$ the $i$-th coordinate of $[P_1](\mathbf s)$. Recall that a Laurent polynomial $f_{\mathbf a}(\mathbf x)=\sum_{i=1}^n a_i\prod_{j=k+1}^n x_j^{p_{ij}}\in\Fq[x_{k+1}^{\pm 1},\ldots,x_n^{\pm 1}]$ is said to be \emph{non-degenerate} with respect to $\Delta$ if, for every face $\tau$ of $\Delta$ that does not contain the origin, the partial derivatives of the polynomial $f_{\mathbf a,\tau}(\mathbf x):=\sum_{i\in I_\tau} a_i\prod_{j=k+1}^n x_j^{p_{ij}}$ do not vanish simultaneously at any non-zero $x_{k+1},\ldots,x_n$, where $I_\tau\subseteq\{1,\ldots,n\}$ is the subset of $i$ such that the $i$-th row of $P_0$ is in $\tau$.

Then \cite[Theorem 0.4]{fu2016} immediately implies:

\begin{theorem}\label{rank}
 Suppose that either $k=n$ (in which case we take $V=\Gm^n$), or there is a non-empty open subset $U$ of $\mathbf a=(a_1,\ldots,a_n)\in \Gm^n$ such that the Laurent polynomial $f_{\mathbf a}(\mathbf x)$ is non-degenerate with respect to $\Delta$, and we take $V=[P_1]^{-1}(U)$. Then $\HH_{A,\chi}$ is a shifted lisse sheaf of rank $d(A)_{\hat p}\cdot v(A)$ on $[Q]^{-1}[\mathbf d](V)$. 
\end{theorem}

\begin{proof}
 If $k<n$, by \cite[Theorem 0.4]{fu2016}, $\GGG_{P_0,\eta}$ is a perverse sheaf of generic rank $v(A)$. By \cite[Proposition 0.1]{fu2009weights} it is actually a shifted sheaf of rank $v(A)$ on $U$. In the terminology of \cite{katz2003semicontinuity}, this sheaf is ``of perverse origin''. By \cite[Proposition 11]{katz2003semicontinuity} this implies that it is lisse, being of constant rank. Therefore, so is $[P_1]^\ast\GGG_{P_0,\eta}$ on $V=[P_1]^{-1}(U)$.
 
Tensoring with $\LL_{\otimes_{j=1}^k\eta_j}$, which is lisse of rank one, does not modify the rank or the property of being lisse. Let $d_i=(d_i)_p(d_i)_{\hat p}$ for $i=1,\ldots,k$. The finite map $[\mathbf d]$ is the composition of a purely inseparable map $(t_1,\ldots,t_k)\mapsto(t_1^{(d_1)_p},\ldots,t_k^{(d_k)_p})$ (which does not change the rank), a finite \'etale map $(t_1,\ldots,t_k)\mapsto(t_1^{(d_1)_{\hat p}},\ldots,t_k^{(d_k)_{\hat p}})$ of degree $d(A)_{\hat p}=(d_1)_{\hat p}\cdots (d_k)_{\hat p}$ (which multiplies the rank by $d(A)_{\hat p}$) and a closed immersion $\Gm^k\to\Gm^r$. Finally, $[Q]$ is an automorphism, and so $[Q]^\ast$ preserves lisseness and rank.

In the $k=n$ case we apply the same argument with $[P]^\ast\LL_{\psi(x_1+\cdots+x_n)}$ instead of $[P_1]^\ast\GGG_{P_0,\eta}$, which is everywhere lisse of rank one.
\end{proof}

By \cite[p.59]{fu2016}, given $A$, for all but a finite set of primes there is such a non-empty open set $U$. However, this does not directly imply that $V=[P_1]^{-1}(U)$ is non-empty, since the image of $[P_1]$ is a proper closed subset of $\Gm^n$. In the next section, we will show that for a fixed $A$ there is indeed an explicit finite set of primes $S$ such that $V$ is non-empty if $p\not\in S$.

\section{Criteria for non-degeneracy}\label{sec-crit}

In this section we will give an explicit sufficient criterion for the non-degeneracy hypothesis of Theorem \ref{rank} to hold. For any subset $I\subseteq\{1,\ldots,n\}$, let $P_I$ (respectively $P_{0,I}$) denote the submatrix of $P$ (resp. of $P_0$) consisting of the rows of $P$ (resp. of $P_0$) whose index is in $I$. Since $P$ is invertible over $\ZZ$, $P_I$ has rank $|I|$ for any $I$, and so does its reduction modulo $p$. If $\tau$ is a face of $\Delta$ and $I_\tau\subseteq\{1,\ldots,n\}$ are the indices for the rows of $P_0$ contained in $\tau$, we write $P_{\tau}$ and $P_{0,\tau}$ for $P_{I_\tau}$ and $P_{0,I_\tau}$ respectively. Note that $P_{0,\tau}$ has rank $\dim(\tau)+1$ (the dimension of the subspace of $\RR^{n-k}$ spanned by $\tau$) if $\tau$ does not contain the origin.

Consider the following condition $(C_p)$: for every face $\tau$ of $\Delta$ that does not contain the origin, the rank of the reduction modulo $p$ of $P_{0,{\tau}}$ is also $\dim(\tau)+1$. This condition is preserved by right multiplication of $P_0$ by an invertible matrix, and so it depends only on $A$, not on the specific chosen basis for its nullspace.

If $k_\tau=\dim(\tau)+1$, then the reduction modulo $p$ of $P_{0,\tau}$ has rank $k_\tau$ if and only if not all $k_\tau\times k_\tau$ minors of $P_{0,\tau}$ are multiples of $p$. In particular, $P_0$ satisfies $(C_p)$ for all but finitely many primes $p$ (namely, for all but those who divide the gcd of the $k_\tau\times k_\tau$ minors of $P_{0,\tau}$ for some face $\tau$ of $\Delta$ not containing the origin). 

\begin{theorem}\label{non-deg} Suppose that $P_0$ satisfies condition $(C_p)$. Then there exists a dense open subset $U\subseteq\Gm^n$ such that
\begin{enumerate}
 \item The Laurent polynomial $f_{\mathbf a}(\mathbf x)=\sum_{i=1}^n a_i \prod_{j=k+1}^n x_j^{p_ {ij}}$ is non-degenerate with respect to $\Delta$ for every $\mathbf a=(a_1,\ldots,a_n)\in U(\overline{\Fq})$.
 \item $V:=[P_1]^{-1}(U)$ is non-empty (so a dense open subset of $\Gm^k$).
 \end{enumerate}
\end{theorem}

\begin{proof} The statement can be proved separately for each of the finitely many faces of $\Delta$. So let $\tau$ be one such face that does not contain the origin, $P_{\tau},P_{0,\tau}$ the corresponding submatrices of $P,P_0$, and let $k_\tau=\dim(\tau)+1$ be the rank of $P_{0,\tau}$ (and of its reduction modulo $p$, by condition $(C_p)$). Let $J\subset I_\tau$ be a maximal subset such that the rows $\mathbf p_{0,j}$ of $P_0$ for $j\in J$ are linearly independent modulo $p$ (and therefore over $\mathbb Z$ too). 

For every $i\in I_\tau\backslash J$, there exist $c_i,c_{ij}\in\ZZ$ for $j\in J$ with $c_i\neq 0$ such that $c_i \mathbf{p}_{0,i}=\sum_{j\in J} c_{ij}\mathbf{p}_{0,j}$. Since all rows of $P_{0,\tau}$ are contained in some hyperplane not containing the origin, evaluating the equation of the hyperplane on this equality of vectors gives $c_i=\sum_{j\in J} c_{ij}$. We may assume that $c_i$ is prime to $p$: otherwise, reducing modulo $p$ would give $\mathbf 0=\sum_{j\in J}\bar c_{ij}\bar{\mathbf p}_{0,j}$ which, by the linear independence of the $\bar{\mathbf p}_{0,j}$, implies that all $c_{ij}$ are divisible by $p$, and we may divide the entire linear combination by $p$.

We need to show that there is a dense open subset $U$ of $\mathbf a\in\Gm^n$ such that the $n-k$ partial derivatives of $f_{\mathbf a,\tau}(\mathbf x)=\sum_{i\in I_\tau} a_i \prod_{j=k+1}^n x_j^{p_ {ij}}$ do not vanish simultaneously on any $(x_{k+1},\ldots,x_n)\in(\overline{\Fq}^\times)^{n-k}$. For $l=k+1,\ldots,n$, we have
\begin{equation}\label{partial}
x_l\frac{\partial f_{\mathbf a,\tau}}{\partial {x_l}}(\mathbf x)=\sum_{i\in I_\tau}a_i \bar p_{il}\prod_{j=k+1}^n x_j^{p_ {ij}}
\end{equation}
Let $y_i=\prod_{j=k+1}^n x_j^{p_{ij}}$ for $i\in J$ and $z_i=\prod_{j=k+1}^n x_j^{p_{ij}}$ for $i\in I_\tau\backslash J$. Then $z_i^{c_i}=\prod_{j\in J} y_j^{c_{ij}}$ for every $i\in I_\tau\backslash J$. If the equations (\ref{partial}) vanish simultaneously for $l=k+1,\ldots,n$ at some non-zero $(x_{k+1},\ldots,x_n)$, then the system of equations
\begin{equation}\label{eq2}
\left\{\begin{array}{ll}
        \sum_{j\in J}a_j\bar p_{jl}y_j+\sum_{i\in I_\tau\backslash J}a_i\bar p_{il}z_i=0 & (l=k+1,\ldots,n) \\
        z_i^{c_i}=\prod_{j\in J} y_j^{c_{ij}} & (i\in I_\tau\backslash J)
       \end{array}
\right.
\end{equation}
has a non-zero solution on $\{y_j,z_i\}$.

Since $\bar p_{il}=\sum_{j\in J}\bar c_i^{-1}\bar c_{ij}\bar p_{jl}$ for every $l$ and $i\in I_\tau\backslash J$, the first equations become
$$\sum_{j\in J}a_j\bar p_{jl}y_j+\sum_{i\in I_\tau\backslash J}a_i\sum_{j\in J}\bar c_i^{-1}\bar c_{ij}\bar p_{jl}z_i=
\sum_{j\in J}\bar p_{jl}\left(a_jy_j+\sum_{i\in I_\tau\backslash J}a_i\bar c_i^{-1}\bar c_{ij}z_i\right)=0
$$
for $l=k+1,\ldots,n$. Given that the rows $\mathbf p_{0,j}$ for $j\in J$ are linearly independent modulo $p$, this implies
$$
a_jy_j+\sum_{i\in I_\tau\backslash J}a_i\bar c_i^{-1}\bar c_{ij}z_i=0
$$
for every $j\in J$. So the system
 \begin{equation}\label{eq3}
\left\{\begin{array}{ll}
        a_jy_j+\sum_{i\in I_\tau\backslash J}a_i\bar c_i^{-1}\bar c_{ij}z_i=0 & (j\in J) \\
        z_i^{c_i}=\prod_{j\in J} y_j^{c_{ij}} & (i\in I_\tau\backslash J)
       \end{array}
\right.
\end{equation}
has a non-zero solution. Via the change of variables $y_j\mapsto a_j^{-1}y_j$, $z_i\mapsto a_i^{-1}z_i$, this becomes
\begin{equation}\label{eq4}
\left\{\begin{array}{ll}
        y_j+\sum_{i\in I_\tau\backslash J}\bar c_i^{-1}\bar c_{ij}z_i=0 & (j\in J) \\
        z_i^{c_i}=a_i^{c_i}\prod_{j\in J}a_j^{-c_{ij}} y_j^{c_{ij}} & (i\in I_\tau\backslash J)
       \end{array}
\right.
\end{equation}

We need to prove that there is a dense open set of $\mathbf a\in\Gm^n$ such that this system has no non-zero solutions. Let $Z\subseteq \Gm^{k_\tau}\times\Gm^{|I_\tau|-k_\tau}\times\Gm^{|I_\tau|-k_\tau}$ (with coordinates $y_j,z_i,t_i$) be the subscheme defined by the equations
\begin{equation}\label{eq5}
\left\{\begin{array}{ll}
        y_j+\sum_{i\in I_\tau\backslash J}c_i^{-1}c_{ij}z_i=0 & (j\in J) \\
        z_i^{c_i}=t_i\prod_{j\in J} y_j^{c_{ij}} & (i\in I_\tau\backslash J)
       \end{array}
\right.
\end{equation}

Then $\dim(Z)=|I_\tau|-k_\tau$ (in fact, $Z$ is isomorphic to an open subset of $\Gm^{|I_\tau|-k_\tau}$ via the projection onto the $z_i$'s). Let $\pi:Z\to\Gm^{|I_\tau|-k_\tau}$ be the projection onto the $t_i$'s. The fibres of $\pi$ are at least one-dimensional: every scalar multiple of a solution of (\ref{eq5}) is again a solution, since $c_i=\sum_{j\in J}c_{ij}$ for every $i\in I_\tau\backslash J$ so the equations (\ref{eq5}) are homogeneous in $\{z_i,y_j\}$. So the dimension of $\pi(Z)$ is at most $|I_\tau|-k_\tau-1$, and in particular there is a dense open set $W=\Gm^{|I_\tau|-k_\tau}\backslash\overline{\pi(Z)}^{Zar}$ of $\mathbf t\in\Gm^{|I_\tau|-k_\tau}$ not in $\pi(Z)$ (that is, such that (\ref{eq5}) has no solutions). We can then take $U=\phi^{-1}(W)$, where $\phi:\Gm^n\to\Gm^{|I_\tau|-k_\tau}$ is the map given by $(a_1,\ldots,a_n)\mapsto(a_i^{c_i}\prod_{j\in J}a_j^{-c_{ij}})_{i\in I_\tau\backslash J}$, which is clearly surjective. 

It remains to show that $[P_1]^{-1}(U)=(\phi[P_1])^{-1}(W)$ is non-empty. We will do that by showing that $\varphi:=\phi[P_1]:\Gm^k\to\Gm^{|I_\tau|- k_\tau}$ is surjective. Note that
$$\varphi(\mathbf s)=\left(\sigma_i(\mathbf s)^{c_i}\prod_{j\in J} \sigma_j(\mathbf s)^{-c_{ij}}\right)_{i\in I_\tau\backslash J}=
\left(\prod_{l=1}^ks_l^{c_ip_{il}-\sum_{j\in J}c_{ij}p_{jl}}\right)_{i\in I_\tau\backslash J}.$$

We claim that the vectors $(c_ip_{i1}-\sum_{j\in J}c_{ij}p_{j1},\ldots,c_ip_{ik}-\sum_{j\in J}c_{ij}p_{jk})\in\ZZ^k$ for $i\in I_\tau\backslash J$ are linearly independent. Indeed, since $c_ip_{il}-\sum_{j\in J}c_{ij}p_{jl}=0$ for every $i\in I_\tau\backslash J$ and every $l=k+1,\ldots,n$, any linear dependence relation among them extends to a linear dependence relation among $c_i\mathbf p_i-\sum_{j\in J}c_{ij}\mathbf p_j\in\ZZ^n$ for $i\in I_\tau\backslash J$, which in turn would give a linear dependence relation among the rows $\mathbf p_1,\ldots,\mathbf p_n$ of $P$, which is impossible since $P$ is invertible. We conclude that the image of $\varphi$ is a subtorus of $\Gm^{|I_\tau|- k_\tau}$ of dimension $|I_\tau|- k_\tau$, that is, $\varphi$ is surjective.
\end{proof}

In the particular case where $k_\tau=|I_\tau|$ (that is, $\tau$ contains exactly $\dim(\tau)+1$ rows of $P_0$) then $J=I_\tau$ and the system (\ref{eq2}) is simply
$$
\sum_{i\in I_\tau}a_i\bar p_{il}y_i=0\;(l=k+1,\ldots,n)
$$
which, since the rows $\mathbf p_{0,i}$ for $i\in J=I_\tau$ are linearly independent modulo $p$, has no non-zero solution for any $\mathbf a$.

\begin{remark}
 Suppose that a codimension one face $\tau$ of $\Delta$ not containing the origin spans the hyperplane $c_{k+1}z_{k+1}+\cdots+c_nz_n=c$ (with $c_1,\ldots,c_n$ relatively prime). Then the vector $(c,\ldots,c)\in\ZZ^{|I_\tau|}$ is a linear combination of the columns of $P_{0,\tau}$. In particular, if $c$ is divisible by $p$, then the columns of $P_{0,\tau}$ are linearly dependent modulo $p$ and therefore the rank of the reduction modulo $p$ of $P_{0,\tau}$ is less than $n-k=\dim(\tau)+1$.
 
 So a necessary condition for $(C_p)$ to be satisfied it that the minimal equation of the hyperplane spanned by any codimension one face of $\Delta$ not containing the origin has prime to $p$ constant term. However, this condition is not sufficient: consider the matrix $A=(2\;-1\;1)\in\mathcal M_{1,3}(\ZZ)$, which clearly has $d(A)=1$. Its nullspace has $\{(0,1,1),(1,1,-1)\}$ as a basis, so 
 $$
 P_0=\begin{pmatrix}
      0 & 1 \\ 1 & 1 \\ 1 & -1 
     \end{pmatrix}
     $$
    and $\Delta\subseteq\RR^2$ is the trapezoid with vertices $(0,0),(0,1),(1,1)$ and $(1,-1)$, with area $3/2$, which gives $v(A)=3$. There are two sides of $\Delta$ not containing the origin, which are contained in the lines with equations $x=1$ and $y=1$, both with constant term $1$. For $\tau$ the side determined by $(1,1)$ and $(1,-1)$ we have
    $$
    P_{0,\tau}=\begin{pmatrix}
      1 & 1 \\ 1 & -1 
     \end{pmatrix}
     $$
     which has rank $1$ modulo $2$. So $P_0$ does not satisfy $(C_2)$. In fact, for $p=2$, $\HH_{A,\chi}$ has generic rank $2$, not $3$ (see example \ref{dim1}).
\end{remark}

The following result gives a more explicit description of the set where $\HH_{A,\chi}$ is guaranteed to be lisse.

\begin{corollary}\label{non-deg-cor} Suppose that $P_0$ satisfies condition $(C_p)$. Let $U\subseteq \Gm^n$ be the largest dense open subset of $(t_1,\ldots,t_n)$ such that the polynomial $\sum_{i=1}^n t_i\prod_{j=k+1}^nx_j^{p_{ij}}$ is non-degenerate with respect to $\Delta$, and let $T\subseteq\Gm^n$ be the subtorus defined by the equations $\prod_{i=1}^n t_i^{p_{ij}}=1$ for $j=k+1,\ldots,n$. Then $W:=[A](T\cap U)$ is a dense open subset of $[A](\Gm^n)$, and $\HH_{A,\chi}$ is a shifted lisse sheaf of rank $d(A)_{\hat p}\cdot v(A)$ on $W$.
\end{corollary}

Note that $U$ and $T$ are defined over $\ZZ[1/S]$, where $S$ is the finite set of primes $p$ such that $
P_0$ does not satisfy condition $(C_p)$. So this proves the existence of the subscheme $Z\subseteq{\mathbb G}^r_{m,\Spec(\ZZ[1/S\ell])}$ stated in Theorem \ref{main}.

\begin{proof} Since the columns of $P_0$ are linearly independent and part of a basis of $\ZZ^n$, $T$ is connected of dimension $k$. Its intersection with the kernel of $[A]$ is finite, since the columns of $P_0$ form a $\ZZ$-basis for the orthogonal complement of the rows of $A$, so the set of rows of $A$ together with them generates a full rank subgroup of $\ZZ^n$. 

By Theorems \ref{rank} and \ref{non-deg}, $\HH_{A,\chi}$ is a shifted lisse sheaf of rank $d(A)_{\hat p}\cdot v(A)$ on $[Q]^{-1}[\mathbf d](V)$, where $V=[P_1]^{-1}(U)\neq\emptyset$. Note that $[\mathbf d]=[D_1]$, where $D_1$ is the matrix consisting of the first $k$ columns of $D$. Since $Q^{-1}D=AP$, $[Q^{-1}][\mathbf d]=[Q^{-1}D_1]=[AP_1]=[A][P_1]$, so $[Q^{-1}][\mathbf d](V)=[A][P_1](V)=[A][P_1][P_1]^{-1}(U)=[A]([P_1](\Gm^k)\cap U)$.

Since $P$ is invertible, the decomposition $P=P_1\sqcup P_0$ induces a decomposition $\Gm^n=[P_1](\Gm^k)\times[P_0](\Gm^{n-k})$. The intersection of $T$ with $[P_0](\Gm^{n-k})$ is finite: since $AP_0=\mathbf 0$, $[P_0](\Gm^{n-k})$ is in the kernel of $[A]$, and $T\cap\ker([A])$ is finite. In particular, the projection $\pi:T\to [P_1](\Gm^k)$ is finite, and therefore surjective by dimension reasons. So $[A]([P_1](\Gm^k)\cap U)=[A](\pi(T)\cap U)$.

Now $\pi(\mathbf t)\in U$ if and only if $\mathbf t\in U$: we have $\mathbf t=\pi(\mathbf t)\cdot\mathbf u$ for some $\mathbf u=[P_0](\mathbf v)\in[P_0](\Gm^{n-k})$. And, by the change of variable $\mathbf x\mapsto\mathbf v\mathbf x$, the polynomial $\sum_i t_i\prod_{j=k+1}^n x_j^{p_{ij}}$ is non-degenerate with respecto to $\Delta$ if and only if $\sum_i t_iu_i\prod_{j=k+1}^n x_j^{p_{ij}}$ is. So $[A](\pi(T)\cap U)=[A](\pi(T\cap U))$.

Finally, $[A](\pi(T\cap U))=[A](T\cap U)$ because, if $\mathbf t\in T$ splits as $\pi(\mathbf t)\cdot \mathbf u$ with $\mathbf u\in[P_0](\Gm^{n-k})\subseteq\ker([A])$, then $[A](\mathbf t)=[A](\pi(\mathbf t)\cdot \mathbf u)=[A](\pi(\mathbf t))$.
\end{proof}
 
We now focus on some particularly nice cases.

\begin{proposition}\label{maximal}
Suppose that $P_0$ satisfies condition $(C_p)$, and every codimension one face of $\Delta$ not containing the origin contains exactly $n-k$ rows of $P_0$. Then $\HH_{A,\chi}$ is a shifted lisse sheaf of rank $d(A)_{\hat p}\cdot v(A)$ on $[A](\Gm^n)$. In particular, if $k=r$, then $\HH_{A,\chi}$ is a shifted lisse sheaf of rank $d(A)_{\hat p}\cdot v(A)$ on $\Gm^r$.
\end{proposition}

\begin{proof}
 The hypothesis implies that any $b$-dimensional face $\tau$ of $\Delta$ that does not contain the origin contains exactly $b+1$ rows of $P_0$, for $0\leq b< n-k-1$. Otherwise, it would contain at least $b+2$ rows, and any $b+1$ dimensional face of $\Delta$ containing $\tau$ but not the origin would contain at least $b+3$ rows. Continuing the process, we would end up with a codimension one face of $\Delta$ not containing the origin and containing at least $n-k+1$ rows of $P_0$.
 
 Let $\tau$ be such a $b$-dimensional face. Then $P_{0,\tau}$ has rank $b+1$ so, applying the comment after Theorem \ref{non-deg}, the partial derivatives of the polynomial $\sum_{i\in I_\tau} \sigma_i(\mathbf s) \prod_{j=k+1}^n x_j^{p_{ij}}$ have no non-zero common solution for any $\mathbf s\in\Gm^k$. That is, the polynomial $\sum_{i=1}^n \sigma_i(\mathbf s) \prod_{j=k+1}^n x_j^{p_ {ij}}$ is non-degenerate with respect to $\Delta$ for every $\mathbf s\in\Gm^k$.
\end{proof}

\begin{proposition}\label{lin-com}
 Suppose that $P_0$ satisfies condition $(C_p)$ and, for every choice of $n-k+1$ columns of $A$, their sum is not a $\QQ$-linear combination of the remaining $k-1$ columns. Then $\HH_{A,\chi}$ is a shifted lisse sheaf of rank $d(A)_{\hat p}\cdot v(A)$ on $[A](\Gm^n)$. In particular, if $k=r$, then $\HH_{A,\chi}$ is a shifted lisse sheaf of rank $d(A)_{\hat p}\cdot v(A)$ on $\Gm^r$.
\end{proposition}

\begin{proof}
We show that any codimension one face of $\Delta$ not containing the origin contains exactly $n-k$ rows of $P_0$, and apply proposition \ref{maximal}.

Suppose, by contradiction, that the $(n-k-1)$-dimensional face $\tau$ contained $n-k+1$ rows of $P_0$; $i_1,\ldots,i_{n-k+1}$. Then
$$
\mathrm{rank}\begin{pmatrix}
     p_{i_1,k+1} & \cdots & p_{i_1,n} & 1 \\
     \vdots & \ddots & \vdots & \vdots \\
     p_{i_{n-k},k+1} & \cdots & p_{i_{n-k},n} & 1 \\
     p_{i_{n-k+1},k+1} & \cdots & p_{i_{n-k+1},n} & 1
    \end{pmatrix}=n-k
$$
since the rows are contained in a hyperplane, and
$$
\mathrm{rank}\begin{pmatrix}
     p_{i_1,k+1} & \cdots & p_{i_1,n} \\
     \vdots & \ddots & \vdots \\
     p_{i_{n-k},k+1} & \cdots & p_{i_{n-k},n} \\
     p_{i_{n-k+1},k+1} & \cdots & p_{i_{n-k+1},n} 
    \end{pmatrix}
    =n-k
$$
since said hyperplane does not contain the origin. So the last column of the first matrix must be a $\QQ$-linear combination of the remaining ones. In other words, there is some element $\mathbf v=(v_1,\ldots,v_n)$ in the nullspace of $A$ with $v_{i_1}=\cdots=v_{i_{n-k+1}}=c\in\ZZ\backslash\{0\}$. If $J=\{1,\ldots,n\}\backslash\{i_1,\ldots,i_{n-k+1}\}$ we then have
$$
c(\mathbf a_{i_1}+\cdots+\mathbf a_{i_{n-k+1}})+\sum_{j\in J}v_j\mathbf a_j=\mathbf 0\Rightarrow
$$
$$
\Rightarrow \mathbf a_{i_1}+\cdots+\mathbf a_{i_{n-k+1}}=-\sum_{j\in J}c^{-1}v_j\mathbf a_j,
$$
in contradiction with the hypothesis.
\end{proof}

Under the hypothesis of the proposition \ref{maximal}, the $(C_p)$ criterion can be easily checked on the matrix $A$:

\begin{proposition}
 Suppose that every codimension one face of $\Delta$ not containing the origin contains exactly $n-k$ rows of $P_0$. Then $P_0$ satisfies condition $(C_p)$ if and only if for every codimension one face $\tau$ of $\Delta$ not containing the origin, the $r\times k$ submatrix of $A$ formed by the $j$-th columns for $j\not\in I_\tau$ has rank $k$ modulo $p$.
\end{proposition}

\begin{proof}
 Any submatrix formed by the $b+1\leq n-k$ rows of $P_0$ contained in some $b$-dimensional face $\tau$ of $\Delta$ not containing the origin has rank $b+1$. So $P_0$ satisfies condition $(C_p)$ if and only if the same happens for the reduction of $P_0$ modulo $p$. It suffices to check this for $b=n-k-1$ (since any face not containing the origin is contained in a codimension one face not containing the origin).
 
 If the condition fails for some face $\tau$ formed by the rows $i_1,\ldots,i_{n-k}$, then some non-trivial linear combination of the columns of $P_0$ has its $i_1,\ldots,i_{n-k}$-th coordinates equal to zero modulo $p$. That is, some element in the nullspace of $A$ has all those coordinates equal to zero modulo $p$ (but not all of the remaining ones), which means the remaining $k$ columns of $A$ are linearly dependent modulo $p$ (and so the rank of the corresponding submatrix of $A$ modulo $p$ is $<k$). Conversely, if the $k$ columns of $A$ are linearly dependent modulo $p$, then any dependence relation among them extends to a linear dependence relation among all columns of $A$ with coefficient $0$ for the $i_1,\ldots,i_{n-k}$-th columns, so there is an element in the nullspace of $A$ modulo $p$ (linear combination of the columns of $P_0$) with those coordinates equal to $0$, which implies the submatrix of $P_0$ formed by those rows has rank $<n-k$ modulo $p$.
\end{proof}

In particular, for $k=r$ a sufficient condition for property $(C_p)$ to hold under this hypothesis is that the determinant of any $r\times r$ submatrix of $A$ is not a non-zero multiple of $p$.

\section{Examples}\label{examples}

\subsection{The case $r=1$}\label{dim1}

This is the classical one-dimensional hypergeometric sheaves which were extensively studied by N. Katz in \cite{katz1990esa}. Let
$$
A=(n_1\cdots n_r\; -m_1\cdots -m_s)
$$
with $n_i,m_i>0$. By \cite[5.6.2]{katz1988gauss}, if $[n]$ denotes the $n$-th power map $\Gm\to\Gm$, then $[n]_!(\LL_\psi\otimes\LL_\chi)[1]$ is, up to a multiplicative translation by $n^{-n}$, geometrically isomorphic to $(\LL_\psi\otimes\LL_{\xi_1})[1]\ast\cdots\ast(\LL_\psi\otimes\LL_{\xi_n})[1]$, where $\xi_1,\ldots,\xi_n$ are the $n$-th roots of $\chi$ if $n$ is a prime to $p$ positive integer; and $[p]_!(\LL_\psi\otimes\LL_\chi)[1]\cong(\LL_\psi\otimes\LL_{\chi^{1/p}})[1]$. We may therefore assume (for the purpose of computing the generic rank) that all columns of $A$ are $(1)$ and $(-1)$:
$$
A=(1\cdots 1 -1\cdots -1)
$$
with $N:=\sum_i(n_i)_{\hat p}$ $1$'s and $M:=\sum_i(m_i)_{\hat p}$ $-1$'s. Under the notation of \cite[Chapter 8]{katz1990esa}, our hypergeometric sheaf $\mathcal H_{A,\chi}$ is isomorphic to $$\chi_{N+1}\cdots\chi_{N+M}(-1)^{deg}\otimes\mathrm{Hyp}_{(-1)^M}(!,\psi;\chi_1,\ldots,\chi_N;\bar\chi_{N+1},\ldots,\bar\chi_{N+M}),$$
since \cite{katz1990esa} uses the conjugates of $\LL_\psi\otimes\LL_{\chi_i}$ for the ``upstairs'' characters, and $\LL_{\bar\psi}\otimes\LL_{\bar\chi_i}=\chi_i(-1)^{deg}\otimes [-1]^\ast(\LL_{\psi}\otimes\LL_{\chi_i})$.

Modulo an inversion of $\Gm$, we can assume that $N>0$. A basis for the nullspace of $A$ is then formed by the vectors $\mathbf e_1-\mathbf e_i$ for $i=2,\ldots,N$ and $\mathbf e_1+\mathbf e_i$ for $i=N+1,\ldots,N+M$. Then
$$
P_0=\begin{pmatrix}
     1 & \cdots & 1 & 1 & \cdots & 1 \\
     -1 &\cdots & 0 & 0 & \cdots & 0 \\ 
     \vdots & \ddots & \vdots & \vdots & \ddots & \vdots \\
     0 & \cdots & -1 & 0 & \cdots & 0 \\
     0 & \cdots & 0 & 1 & \cdots & 0 \\
     \vdots & \ddots & \vdots & \vdots & \ddots & \vdots \\
     0 & \cdots & 0 & 0 & \cdots & 1 
    \end{pmatrix},
$$
which satisfies the $(C_p)$ condition for all primes $p$ (as all its $(N+M-1)\times(N+M-1)$ minors are $\pm 1$). Then the rank of $\HH_{A,\chi}$ and the open set $U$ where it is lisse are the same for all $p$ (and for any choice of characters $\chi_i$). We can then pick any $p>N+M$ and apply the previous isomorphism again (in the other direction) to reduce to the case $A=(N\;-M)$.

If $M=0$ we can take $A=(N)$, so the nullspace is trivial and $\HH_{A,\chi}$ is lisse on $\Gm$ of rank $d(A)_{\hat p}=N$. Similarly, if $N=0$ then $\HH_{A,\chi}$ is lisse on $\Gm$ of rank $M$. If both $N$ and $M$ are non-zero, let $D=\gcd(N,M)$. Then the nullspace of $A$ is generated by $(M/D\; N/D)$, and $d(A)_{\hat p}=D_{\hat p}=D$. The polytope $\Delta$ is now the interval $[0,\sup(N,M)/D]$, with volume $\sup(N,M)/D$. We conclude that the generic rank of $\HH_{A,\chi}$ is $\sup(N,M)$, as it was known from \cite[Theorem 8.4.2]{katz1990esa}.

The hypothesis of proposition \ref{lin-com} in this case (where $n-k+1=n$) says simply that $N-M\neq 0$, that is, that $N\neq M$. In that case, proposition \ref{lin-com} states that $\mathcal H_{A,\chi}$ is a shifted lisse sheaf on $\Gm$, as again it is known by \cite[Theorem 8.4.2]{katz1990esa} (in fact, in this case, the condition is also necessary).

In the remaining case $N=M$, we have
$$
P_0=\left(\begin{array}{c}
1   \\
1   \\
\end{array}\right)
$$
and the polynomial we need to test for non-degeneracy in corollary \ref{non-deg-cor} is $(t_1+t_2)x$, which clearly is degenerate with respect to $\Delta=[0,1]$ if and only if $t_1+t_2=0$. The subtorus $T$ is here $t_1t_2=1$, so $T\cap W=\{(t,t^{-1})|t^2\neq -1\}$. Its image by $[A]$ is then $\{t^N/t^{-N}|t^2\neq -1\}=\{t^{2N}|t^2\neq -1\}=\Gm\backslash\{(-1)^N\}$. So $\mathcal H_{A,\chi}$ is a shifted lisse sheaf of rank $N$ outside the point $(-1)^N$, again matching the known results in \cite[Theorem 8.4.2]{katz1990esa}.  Note that, in this case, the multiplicative translations needed to go from $(1\cdots 1\;-1\cdots -1)$ to $(N\; -N)$ cancel each other out, so the singular point does not move. 

\subsection{The case $k=n$} Now $[A]$ is a finite map, composition of a purely inseparable map of degree $d(A)_p$, a finite \'etale map of degree $d(A)_{\hat p}$ and a closed immersion onto the subtorus $T$ of $\Gm^r$ image of $\Gm^n$ under the homomorphism $[B]$ defined by the first $n$ columns of $Q^{-1}$. In particular, $\HH_{A,\chi}$ is a shifted lisse sheaf on $T$ of rank $d(A)_{\hat p}$.

This can be generalized as follows, since convolution commutes with direct image under a torus homomorphism: suppose that $\vv_1,\ldots,\vv_n\in\ZZ^r$ are primitive and linearly independent, and the columns of $A$ consist of $m_i$ copies of $\vv_i$ and $n_i$ copies of $-\vv_i$ for $i=1,\ldots,n$, with $m_i+n_i>0$. Let $\chi=\otimes_{i,j}\chi_{ij}$ with $i=1,\ldots,n$, $j=1,\ldots,m_i+n_i$. Let $\HH_i$ be the one-dimensional hypergeometric associated to the matrix $(1\cdots 1\;-1\cdots -1)$, with $m_i$ $1$'s and $n_i$ $-1$'s, and the character $\chi_i=\otimes_j\chi_{ij}$.

Then $\HH_{A,\chi}$ is a shifted constructible sheaf on $T$ of generic rank $d(A)_{\hat p}\cdot\prod_i\sup(m_i,n_i)$. More precisely, if $C$ is the matrix $(\vv_1|\cdots|\vv_n)$ (which has rank $n$ and the same $d$ as $A$), then $\HH_{A,\chi}=\R[C]_!(\HH_1\boxtimes\cdots\boxtimes\HH_n)$. In particular, $\HH_{A,\chi}$ is lisse on the complement in $T$ of the image by $[C]$ of the union of the subtori of $\Gm^n$ defined by $x_i=(-1)^{n_i}$ for all $i\in\{1,\ldots,n\}$ such that $m_i=n_i$.

\subsection{The case $k=r=n-1$} That is, $A$ is an $r\times(r+1)$ matrix of rank $r$. Let $A_i$ be the $r\times r$ matrix obtained by removing the $i$-th column of $A$, and $\delta_i=(-1)^{i-1}\det(A_i)/d(A)$. Then the right nullspace of $A$ is generated by $(\delta_1,\ldots,\delta_n)$: indeed, it is a rank one subgroup of $\ZZ^n$, so any primitive element in it forms a basis. The vector $(\det(A_1),-\det(A_2),\ldots,(-1)^{n-1}\det(A_n))$ is in the nullspace, since the dot product with any row of $A$ is the determinant of a matrix with two identical rows, and therefore 0; and the greatest common divisor of its elements is $d(A)$.

Let $\delta^+=\max(\delta_i)$ and $\delta^-=\min(\delta_i)$. The polytope $\Delta$ is $[0,\delta^+]$ if all non-zero $\delta_i$ are positive, $[\delta^-,0]$ if all non-zero $\delta_i$ are negative, and $[\delta^-,\delta^+]$ if there are positive and negative $\delta_i$'s. Condition $(C_p)$ holds if $\delta^+$ (resp. $\delta^-$, both $\delta^+$ and $\delta^-$) is prime to $p$. In that case, the generic rank of $\HH_{(A,\chi)}$ is $\max(0,\delta^+)-\min(0,\delta^-)$.

If $\delta^+>0$, let $I=\{i=1,\ldots,n|\delta_i=\delta^+\}$. If $\delta^-<0$, let $J=\{i=1,\ldots,n|\delta_i=\delta^-\}$. Then the polynomial $\sum_{i=1}^n \sigma_i(\mathbf s) \prod_{j=k+1}^n x_j^{p_ {ij}}=\sum_{i=1}^n \sigma_i(\mathbf s) x_n^{\delta_i}$ is non-degenerate with respect to $\Delta$ if and only if $\sum_{i\in I}s_1^{i1}\cdots s_r^{ir}\neq 0$ if $\delta^+>0$ and $\sum_{i\in J}s_1^{i1}\cdots s_r^{ir}\neq 0$ if $\delta^-<0$. In particular, if $I$ and $J$ have at most one element then $\HH_{A,\chi}$ is a shifted lisse sheaf on the entire $\Gm^r$.

\section{Hypergeometric sheaves in characteristic 0}\label{lift}

If the matrix $P_0$ satisfies the (transposed) ``non-confluence'' condition of \cite[Theorem 0.8]{fu2016}, that is, the vector $(1,1,\ldots,1)$ is a linear combination of its columns, then the theorem implies that the corresponding GKZ-hypergeometric sheaf is liftable to characteristic zero and, in particular, has tame ramification for generic $p$. Since the hypergeometric $\HH_{A,\chi}$ can be derived from it by operations that are defined in characteristic zero, the same can be said about it.

This condition is equivalent to the fact that $(1,1,\ldots,1)$ is in the nullspace of $A$, that is, that the sum of the columns of $A$ is the zero vector. In this section, we will give an independent proof of the liftability to characteristic 0 of the hypergeometric sheaf $\HH_{A,\chi}$ in that case.

For $N\in\ZZ$ prime to $p$, let $R=\ZZ[\mu_N][1/N]$ be the ring of integers of the $N$-th cyclotomic extension of $\QQ$, with the primes dividing $N$ inverted. It is an \'etale $\ZZ[1/N]$-algebra. Let $\mathfrak p\in\Spec(R)$ be a prime over $p$. Then $R/\mathfrak p\cong\mathbb F_{p^f}$, where $f$ is the smallest positive integer such that $N|p^f-1$. If $N|q-1$, every inclusion $R/\mathfrak p\hookrightarrow \Fq$ induces a specialization map $\Spec(\Fq)\to\Spec(R)$.

\begin{theorem}
 Suppose that the columns of $A$ add up to $\mathbf 0$. Then
 
 \begin{enumerate}
  \item There exist objects $\mathcal K_i$ in $\Dbc({\mathbb G}^r_{m,R[1/\ell]},\Ql)$ for $i=1,\ldots,N^{n-1}$ such that for every finite field $\Fq$ with $(\ell,q)=1$ and $N|q-1$, every specialization $\Spec(\Fq)\to\Spec(R)$ and every character $\chi=\otimes_{i=1}^n\chi_i:(\Fq^\times)^n\to\Ql^\times$ such that $\chi^N=\mathbf 1$ and $\chi_1\cdots\chi_n$ is trivial, $\HH(A,\chi)\cong\mathcal K_i|_{{\mathbb G}^r_{m,\Fq}}$ for some $i$.
  \item There exist objects $\mathcal M_i$ in $\Dbc({\mathbb G}^r_{m,R[1/\ell]},\Ql)$ for $i=1,\ldots,N^{n-1}(N-1)$ such that for every finite field $\Fq$ with $(\ell,q)=1$ and $N|q-1$, every specialization $\Spec(\Fq)\to\Spec(R)$ and every character $\chi=\otimes_{i=1}^n\chi_i:(\Fq^\times)^n\to\Ql^\times$ such that $\chi^N=\mathbf 1$ and $\chi_1\cdots\chi_n$ is not trivial, $\HH(A,\chi)\cong\mathcal M_i|_{{\mathbb G}^r_{m,\Fq}}\otimes g(\psi,\chi_1\cdots\chi_n)^{deg}$ for some $i$, where $$g(\psi,\chi_1\cdots\chi_n)=-\sum_{x\in\Fq^\times}\psi(x)\chi_1\cdots\chi_n(x)$$ is the Gauss sum and $\alpha^{deg}$ is the geometrically trivial sheaf on which Frobenius acts by multiplication by $\alpha$.
 \end{enumerate}

\end{theorem}

\begin{proof}
First of all, note that given a specialization $\Spec(\Fq)\to\Spec(R)$, any Kummer sheaf $\LL_{\chi}$ on ${\mathbb G}^n_{m,\Fq}$ with $\chi^N=\mathbf 1$ can be lifted to a lisse sheaf on ${\mathbb G}^n_{m,R[1/\ell]}$: the \'etale cover $[N]:\Gm^n\to\Gm^n$ given by the $N$-th power map is Galois with abelian Galois group $G\cong \mu_N\times\cdots\times\mu_N$ over both $R$ and $\Fq$. In the latter case, $G$ can be identified with a subgroup of $(\Fq^\times)^n$.

Then $\R[N]_*\Ql\cong\bigoplus_{\chi}\LL_\chi$, where the sum is taken over the $\Ql$-valued characters of $G$ and $\LL_\chi$ is the Kummer lisse sheaf on $\Gm^n$ obtained by pulling back $\chi$ to $\pi_1(\Gm^n,\eta)$ (where $\eta$ is a generic geometric point) via the quotient $\pi_1(\Gm^n,\eta)\twoheadrightarrow G$ determined by the finite \'etale cover $[N]$. This decomposition over $R[1/\ell]$ restricts to a similar decomposition over $\Fq$.

But, on the other hand, over $\Fq$ we have the usual decomposition $\R[N]_*\Ql\cong\bigoplus_{\chi}\LL_\chi$ as a direct sum of Kummer sheaves, where the sum is now taken over the set of characters $\chi:(\Fq^\times)^n\to\Ql^\times$ with trivial $N$-th power. Comparing both decompositions, we conclude that every such Kummer sheaf is isomorphic to the restriction to $\Fq$ of a Kummer sheaf defined over $R[1/\ell]$.

In particular, for every morphism $f:X\to\Gm^n$ defined over $R[1/\ell]$ and every character $(\Fq^\times)^n\to\Ql^\times$ with trivial $N$-th power, the sheaf $\LL_{\chi(f_{\Fq})}:=f_{\Fq}^\ast\LL_\chi$ on $X_{\Fq}$ is isomorphic to the restriction to $\Fq$ of a sheaf on $X$.

We can split the set of characters of $\mu_N\times\cdots\times\mu_N$ into two disjoint subsets $\mathcal A$ and $\mathcal B$ of cardinalities $N^{n-1}$ and $N^{n-1}(N-1)$ respectively according to whether their restriction to the diagonal subgroup of $\mu_N^n$ is trivial or not (equivalently, whether the restriction of the associated Kummer sheaf to the diagonal $\Gm\hookrightarrow\Gm^n$ is trivial). Under the above correspondence, Kummer sheaves for characters in $\mathcal A$ and $\mathcal B$ correspond to Kummer sheaves for characters $\chi=\otimes_{i=1}^n\chi_i:(\Fq^\times)^n\to\Ql^\times$ such what $\chi_1\cdots\chi_n$ is trivial (resp. not trivial).

 For a given $\Fq$ with an embedding $\Spec(\Fq)\to\Spec(R)$ and a character $\chi:(\Fq^\times)^n\to\Ql^\times$, we have 
 $$
\mathcal H_{A,\chi}=\R[A]_!(\LL_{\psi(x_1+\cdots+x_n)}\otimes\LL_\chi)[n].$$
Let $\phi:\Gm^n\to\Gm^n$ be the automorphism given by the invertible matrix
$$
U=\begin{pmatrix}
 1 & 0 & \cdots & 0 & 1 \\
 0 & 1 & \cdots & 0 & 1 \\
 \vdots & \vdots & \ddots & \vdots & \vdots \\
 0 & 0 & \cdots & 1 & 1 \\
 0 & 0 & \cdots & 0 & 1
\end{pmatrix}
$$
that is,
$$(x_1,\ldots,x_{n-1},x_n)\mapsto(x_1x_n,\ldots,x_{n-1}x_n,x_n),$$
then
$$\HH_{A,\chi}=\R([A]\phi)_!\phi^*(\LL_{\psi(x_1+\cdots+x_n)}\otimes\LL_{\chi_1(x_1)}\otimes\cdots\otimes\LL_{\chi_{n-1}(x_{n-1})}\otimes \LL_{\chi_{n}(x_{n})})[n]\cong
$$
$$
\cong \R[AU]_!(\LL_{\psi((x_1+\cdots+x_{n-1}+1)x_n)}\otimes\LL_{\chi_1(x_1)}\otimes\cdots\otimes\LL_{\chi_{n-1}(x_{n-1})}\otimes \LL_{\chi_1\cdots\chi_{n}(x_{n})})[n].
$$
Let $B$ be the matrix obtained from $A$ by removing the last column, and note that $AU=B\sqcup(\mathbf 0)=B(I_{n-1}\sqcup(\mathbf 0))$. Then $[AU]=[B][I_{n-1}\sqcup(\mathbf 0)]=[B]\pi$, where $\pi:\Gm^n\to\Gm^{n-1}$ is the projection onto the first $n-1$ coordinates. So
\begin{equation}\label{RB}
\HH_{A,\chi}\cong \R[B]_!\R\pi_!(\LL_{\psi((x_1+\cdots+x_{n-1}+1)x_n)}\otimes\LL_{\tilde\chi}\otimes \LL_{\chi_1\cdots\chi_{n}(x_{n})})[n]\cong
\end{equation}
$$
\cong
\R[B]_!\left(\LL_{\tilde\chi}\otimes R\pi_!(\LL_{\psi((x_1+\cdots+x_{n-1}+1)x_n)}\otimes\LL_{\chi_1\cdots\chi_{n}(x_{n})})\right)[n]
$$
by the projection formula, where $\tilde\chi:(\Fq^\times)^{n-1}\to\Ql^\times$ is the character $(x_1,\ldots,x_{n-1})\mapsto \chi_1(x_1)\cdots\chi_{n-1}(x_{n-1}).$

Let $\iota_Y:Y\hookrightarrow \Gm^n$ be the inclusion of the closed subset defined by $x_1+\cdots+x_{n-1}+1=0$, and $\iota_Z:Z\hookrightarrow\Gm^n$ its open complement. We have an exact triangle
\begin{equation}\label{triangle}
\R(\pi\iota_Z)_!\iota_Z^*(\LL_{\psi((x_1+\cdots+x_{n-1}+1)x_n)}\otimes\LL_{\chi_1\cdots\chi_{n}(x_{n})})\to
\end{equation}
$$
\to \R\pi_!(\LL_{\psi((x_1+\cdots+x_{n-1}+1)x_n)}\otimes\LL_{\chi_1\cdots\chi_{n}(x_{n})})\to
$$
$$
\to \R(\pi\iota_Y)_!\iota_Y^*(\LL_{\psi((x_1+\cdots+x_{n-1}+1)x_n)}\otimes\LL_{\chi_1\cdots\chi_{n}(x_{n})}) \to.
$$
Suppose that $\chi_1\cdots\chi_n$ is non-trivial. On $Y$ $\LL_{\psi((x_1+\cdots+x_{n-1}+1)x_n)}$ is trivial, so 
$$
\R(\pi\iota_Y)_!\iota_Y^*(\LL_{\psi((x_1+\cdots+x_{n-1}+1)x_n)}\otimes\LL_{\chi_1\cdots\chi_{n}(x_{n})})=
$$
$$
=\R(\pi\iota_Y)_!(\LL_{\chi_1\cdots\chi_{n}(x_{n})|Y})=0,
$$
since $\LL_{\chi_1\cdots\chi_n(x_n)}$ has trivial cohomology on $\Gm$. So we get an isomorphism in $\Dbc(\Gm^r,\Ql)$:
$$
\R(\pi\iota_Z)_!\iota_Z^*(\LL_{\psi((x_1+\cdots+x_{n-1}+1)x_n)}\otimes\LL_{\chi_1\cdots\chi_{n}(x_{n})})\cong
$$
$$
\cong \R\pi_!(\LL_{\psi((x_1+\cdots+x_{n-1}+1)x_n)}\otimes\LL_{\chi_1\cdots\chi_{n}(x_{n})}).
$$
Now on $Z$ we consider the automorphism $\alpha:(x_1,\ldots,x_{n-1},x_n)\mapsto(x_1,\ldots,x_{n-1},x_n(x_1+\cdots+x_{n-1}+1)^{-1})$, which preserves the fibres of $\pi\iota_Z$, and we get
$$
\R(\pi\iota_Z)_!\iota_Z^*(\LL_{\psi((x_1+\cdots+x_{n-1}+1)x_n)}\otimes\LL_{\chi_1\cdots\chi_{n}(x_{n})})\cong
$$
$$
\cong \R(\pi\iota_Z)_!(\LL_{\psi(x_n)}\otimes\LL_{\chi_1\cdots\chi_{n}(x_{n}(x_1+\cdots+x_{n-1}+1)^{-1})})\cong
$$
$$
\cong \LL_{\overline{\chi_1\cdots\chi_n}(x_1+\cdots+x_{n-1}+1)}\otimes \R(\pi\iota_Z)_!(\LL_{\psi(x_n)}\otimes\LL_{\chi_1\cdots\chi_{n}(x_{n})})\cong
$$
$$
\cong
\LL_{\overline{\chi_1\cdots\chi_n}(x_1+\cdots+x_{n-1}+1)}\otimes \R\Gamma_c(\Gm,\LL_\psi\otimes\LL_{\chi_1\cdots\chi_n})=
$$
$$
=\LL_{\overline{\chi_1\cdots\chi_n}(x_1+\cdots+x_{n-1}+1)}\otimes g(\psi,\chi_1\cdots\chi_n)^{deg}[-1].
$$
We conclude that
$$
\HH_{A,\chi}\cong \R[B]_!(\LL_{\tilde\chi}\otimes\LL_{\overline{\chi_1\cdots\chi_n}(x_1+\cdots+x_{n-1}+1)})\otimes g(\psi,\chi_1\cdots\chi_n)^{deg}[n-1].
$$

If $\chi_1\cdots\chi_n$ is trivial, then (\ref{RB}) becomes
$$
\HH_{A,\chi}\cong\R[B]_!\left(\LL_{\tilde\chi}\otimes \R\pi_!(\LL_{\psi((x_1+\cdots+x_{n-1}+1)x_n)})\right)[n].
$$

Let $\bar\pi:\Gm^{n-1}\times\AAA^1\to\Gm^{n-1}$ and $\pi_0:\Gm^{n-1}\times\{0\}\to\Gm^{n-1}$ be the projections (note that $\pi_0$ is an isomorphism). We have exact triangles
$$
\R\pi_!(\LL_{\psi((x_1+\cdots+x_{n-1}+1)x_n)})\to\R\bar\pi_!(\LL_{\psi((x_1+\cdots+x_{n-1}+1)x_n)}) 
\to
$$
$$
\to\R\pi_{0!}(\LL_{\psi((x_1+\cdots+x_{n-1}+1)x_n)})=\R\pi_{0!}\Ql\cong\Ql\to.
$$
and (with $Z$ and $Y$ subsets of $\Gm^{n-1}\times\AAA^1$ defined as above)
$$
\R(\bar\pi\iota_Z)_!\iota_Z^*(\LL_{\psi((x_1+\cdots+x_{n-1}+1)x_n)})
\to \R\bar\pi_!(\LL_{\psi((x_1+\cdots+x_{n-1}+1)x_n)})\to
$$
$$
\to \R(\bar\pi\iota_Y)_!\iota_Y^*(\LL_{\psi((x_1+\cdots+x_{n-1}+1)x_n)})=\R(\bar\pi\iota_Y)_!\Ql\to.
$$
Applying the automorphism $\alpha$ again on $Z$, we get
$$
\R(\bar\pi\iota_Z)_!\iota_Z^*(\LL_{\psi((x_1+\cdots+x_{n-1}+1)x_n)})\cong\R(\bar\pi\iota_Z)_!\LL_{\psi(x_n)}=0,
$$
since the cohomology of $\LL_\psi$ on $\AAA^1$ is trivial. So we get an isomorphism
$$
\R\bar\pi_!(\LL_{\psi((x_1+\cdots+x_{n-1}+1)x_n)})\cong\R(\bar\pi\iota_Y)_!\Ql
$$
and an exact triangle
$$
\R\pi_!(\LL_{\psi((x_1+\cdots+x_{n-1}+1)x_n)})\to\R(\bar\pi\iota_Y)_!\Ql 
\to
\R\pi_{0!}\Ql\cong\Ql\to.
$$
where the second morphism is the one from the excision triangle for the split $Y\backslash(\Gm^{n-1}\times\{0\})\hookrightarrow Y\hookleftarrow \Gm^{n-1}\times\{0\}$. So $\R\pi_!(\LL_{\psi((x_1+\cdots+x_{n-1}+1)x_n)})\cong{\mathcal G}[-1]$, where $\mathcal G$ is the mapping cone of the excision morphism $\R(\bar\pi\iota_Y)_!\Ql 
\to
\R\pi_{0!}\Ql$, and $\HH_{A,\chi}=\R[B]_!(\LL_{\tilde\chi}\otimes{\mathcal G})[n-1]$.

As a conclusion, we can take as $\mathcal K_i$ the $\R[B]_!(\LL_{\tilde\chi}\otimes{\mathcal G})[n-1]$ for the $N^{n-1}$ characters $\chi$ in $\mathcal A$, and as $\mathcal M_i$ the $\R[B]_!(\LL_{\tilde\chi}\otimes\LL_{\overline{\chi_1\cdots\chi_n}(x_1+\cdots+x_{n-1}+1)})[n-1]$ for the $N^{n-1}(N-1)$ characters $\chi$ in $\mathcal B$.
\end{proof}

\bibliographystyle{amsalpha}
\bibliography{bibliography}

\end{document}